\documentclass[11pt]{amsart}
\usepackage{amsfonts}
\usepackage{amsmath}
\usepackage{amsthm}
\usepackage{url}
\usepackage[all]{xy}
\usepackage{latexsym}
\usepackage{amssymb}
\usepackage[cp850]{inputenc}
\usepackage{amsfonts}
\usepackage{graphicx}

\usepackage[abs]{overpic}

\usepackage[usenames]{color}

\newtheorem{sat}{Theorem}[section]		
\newtheorem{lem}[sat]{Lemma}
\newtheorem*{claim}{Claim}
			
\newtheorem{prop}[sat]{Proposition}

\newtheorem*{defi*}{Definition}			
\newtheorem*{bei*}{Example}
\newtheorem*{sat*}{Theorem}				
\newtheorem*{kor*}{Corollary}
\newtheorem*{rmk*}{Remark}				
	
\newtheorem*{quest*}{Question}	
\newtheorem{fact}{Fact}	

\let\ssection=\section
\renewcommand{\section}{\setcounter{equation}{0}\ssection}

\newtheorem*{namedtheorem}{\theoremname}
\newcommand{\theoremname}{testing}
\newenvironment{named}[1]{\renewcommand{\theoremname}{#1}\begin{namedtheorem}}{\end{namedtheorem}}

\theoremstyle{remark}
\newtheorem*{bem}{Remark}

\newtheorem*{namedtheoremr}{\theoremnamer}
\newcommand{\theoremnamer}{testing}

			\newcommand{\BH}{\mathbb H}
\newcommand{\BR}{\mathbb R}			
\newcommand{\BN}{\mathbb N}			
\newcommand{\BS}{\mathbb S}			\newcommand{\BZ}{\mathbb Z}

		\newcommand{\CB}{\mathcal B}

\newcommand{\D}{\partial}
\newcommand{\DD}{\nabla}

\newcommand{\Mod}{\mathrm{Mod}}

\DeclareMathOperator{\Isom}{Isom}	
\DeclareMathOperator{\vol}{vol}		

\DeclareMathOperator{\diam}{diam}

\newcommand{\comment}[1]{}

\DeclareMathOperator{\Dome}{Dome}
\DeclareMathOperator{\BCG}{Bar}
\DeclareMathOperator{\Center}{Center}

\newcommand{\norm}[1]{\lVert #1 \rVert}

\newcommand{\BB}{\mathbb{B}}

\begin{document}

\title[]{Harmonic extensions of quasiregular maps}
\author{Pekka Pankka}
\address{Department of Mathematics and Statistics, University of Helsinki}
\email{pekka.pankka@helsinki.fi}
\author{Juan Souto}
\address{Univ Rennes, CNRS, IRMAR - UMR 6625, F-35000 Rennes, France}
\email{juan.souto@univ-rennes1.fr}
\date{\today}

\thanks{P.P. was partially supported by the Academy of Finland project \#297258.}
\thanks{J.S. benefits from the support of the French government {\em Investissements d'Avenir} program ANR-11-LABX-0020-01.}
\thanks{Both authors were partially supported by the Projet international de coop\'eration scietifique PICS 7734.}

\begin{abstract}
We prove that every non-constant quasiregular selfmap of the $n$-sphere $\BS^{n}$ admits a harmonic extension to the hyperbolic space $\BH^{n+1}$ for $n\ge 2$. 
\end{abstract}

\maketitle


\section{Introduction}

A continuous self-map $f\colon \BS^n \to \BS^n$ of the $n$-sphere $\BS^n$, for $n\ge 2$, is \emph{quasiregular} if it belongs to the Sobolev space $W^{1,n}_\mathrm{loc}(\BS^n, \BS^n)$ and there exists a constant $K\ge 1$ for which 
\begin{equation}
\label{eq:K}
\norm{Df}^n \le K\cdot \det(Df)
\end{equation}
almost everywhere in $\BS^n$. In this paper we prove that, after identifying $\BS^n$ with the boundary at infinity of hyperbolic space $\BH^{n+1}$, quasiregular maps admit harmonic extensions to $\BH^{n+1}$:

\begin{sat}\label{harmonic}
Let $n\ge 2$ and $f\colon \BS^n \to \BS^n$ a non-constant quasiregular map. Then there exists a harmonic map $H_f \colon \BH^{n+1}\to \BH^{n+1}$ extending $f$.
\end{sat}

Note at this point that quasiregular maps are closely related to quasiconformal maps. In fact, a quasiconformal map is nothing other than a quasiregular homeomorphism. The existence of harmonic extensions of quasiconformal maps was conjectured by Schoen \cite{Schoen} and Li--Wang \cite{Li-Wang}, and was proved by Markovic \cite{Markovic-Invent,Markovic-JAMS} in low-dimensions and by Lemm--Markovic \cite{Lemm-Markovic} in general. A different solution, albeit with certain similarities, was given then by Benoist and Hulin \cite{Benoist-Hulin} who proved the existence of harmonic maps at bounded distance of quasi-isometric embeddings of rank one symmetric spaces, allowing thus the dimension of the target to be larger than that of the domain. In fact, to prove Theorem \ref{harmonic} we follow the lines of the argument used by Benoist and Hulin \cite{Benoist-Hulin}, and we strongly recommend the reader to study that beautifully written paper before working through the details of the proof of Theorem \ref{harmonic}.

As we just said, we will follow the argument of Benoist and Hulin. In doing so we find ourself from the outset with an obvious difficulty. In their setting, Benoist and Hulin start with a quasi-isometry and then construct a harmonic map at bounded distance from that quasi-isometry. However, a typical quasiregular map $\BS^n \to \BS^n$ has degree greater than 1 and hence does not admit a quasi-isometric extension to hyperbolic space. In fact, quasiconformal maps $\BS^n \to \BS^n$ are precisely those quasiregular maps which extend to a quasi-isometry $\BH^{n+1}\to \BH^{n+1}$. 

To bypass this difficulty we consider Besson, Courtois and Gallot's barycentric extension $F_f:\BH^{n+1}\to\BH^{n+1}$ of $f$, and one could actually summarise the proof of Theorem \ref{harmonic} as follows: \emph{we run the Benoist-Hulin argument using the barycentric extension $F_f$ in lieu of their quasi-isometry.} It turns out that to do so we need to establish following 3 facts on the barycentric extension:
\begin{enumerate}
\item The map $F_f$ is uniformly Lipschitz and has uniformly bounded second derivative.
\item The map $F_f$ does not map sets of large visual volume to sets of small visual diameter. 
\item The map $F_f$ behaves like a quasi-isometry in most directions.
\end{enumerate}
These three facts are the content of Proposition \ref{prop-lipschitz}, Proposition \ref{prop-volume}, and Proposition \ref{prop-radialQI} below, and the main work of this paper is to establish them. The proofs use very classical tools such as the relation between modulus of annuli in $\BS^n$ to distances in $\BH^{n+1}$, together with a compactness result that might be of independent interest and which we state therefore here in the introduction:

\begin{prop}\label{prop-compactness}
Let $n\ge 2$, $K\ge 1$, and $d\in \BN$. Then, given a sequence $(f_m)$ of $K$-quasiregular mappings $\BS^n \to \BS^n$ of degree at most $d$, there exists a set $P\subset \BS^n$ of cardinality at most $d$ and a subsequence $(f_{m_k})$ converging locally uniformly on $\BS^n\setminus P$ to a $K$-quasiregular mapping $f\colon \BS^n \to \BS^n$ of degree at most $d$. 
\end{prop}

\begin{bem}
The Besson-Courtois-Gallot barycentric extension is closely related to the so-called Douady-Earle extension \cite{Douady-Earle}. The latter is more commonly used in the field of quasiconformal analysis. See \cite{Hu-Muzician} for a recent application. See also the references therein. 
\end{bem}

The paper is organised as follows. After establishing some notation used through the paper, we prove Theorem \ref{harmonic} in Section \ref{sec-harmonic} assuming the properties (1), (2) and (3) above on the barycentric extension. In Section \ref{sec-quasiregular} we recall basic facts about quasiregular maps. Proposition \ref{prop-compactness} is proved in that section. In Section \ref{sec-BCG} we recall the construction and some basic properties of the barycentric extension, delaying the proof of (1), (2) and (3) to Section \ref{sec-meat}.  We conclude with a discussion of the analogue of Theorem \ref{harmonic} in dimension $n=1$.

\subsection*{Acknowledgements}
We thank Aimo Hinkkanen and Vlad Markovic for discussions on these topics.

\section*{Notation.}
We fix the notation that we are going to use in the sequel. First, we identify the $n$-sphere $\BS^n$ with the boundary at infinity $\D_\infty\BH^{n+1}$ of the hyperbolic space $\BH^{n+1}$ of one dimension higher. We denote then by $\bar\BH^{n+1}=\BH^{n+1}\cup\BS^n$ the compactification of hyperbolic space by its boundary at infinity. Given distinct points $x,y\in\bar\BH^{n+1}$, let $[x,y]$ be the unique (possibly infinite) geodesic arc in $\BH^{n+1}$ joining them. Let also
$$\pi_{[x,y]}:\bar\BH^{n+1}\to[x,y]$$
be the convex projection onto this arc. 

For each $x\in \BH^{n+1}$ and $\theta, \eta\in \BS^n=\D_\infty \BH^{n+1}$, we let $d_x(\theta,\eta)$ to be the angle between the unit vectors $n_\theta(x),n_\eta(x)\in T^1_x\BH^{n+1}$ with 
$$\lim_{t\to\infty}\exp_x(t\cdot n_\theta(x))=\theta\text{ and }\lim_{t\to\infty}\exp_x(t\cdot n_\eta(x))=\eta.$$
This determines, at each point $x\in\BH^{n+1}$, a metric $d_x$ on $\BS^n$. We then denote by $\diam_x(X)$ the $d_x$-diameter of a set $X\subset\BS^n$. Also, we let $\vol_x$ be the unique probability measure on $\BS^n$ (that is, $\vol_x(\BS^n)=1$) 
invariant under the isometries of $(\BS^n,d_x)$. The distance $d_x$ and the related notions $\diam_x$ and $\vol_x$ behave well with respect to the group of isometries of $\BH^{n+1}$ in the sense that 
$$d_{\phi(x)}(\phi(\theta),\phi(\eta))=d_x(\theta,\eta)$$
for all $\phi\in\Isom(\BH^{n+1})$, all $x\in\BH^{n+1}$, and all $\theta$ and $\eta$ in $\BS^n$. It follows that all the metrics $d_x$ are in the same conformal class.

We fix a base point $o\in\BH^{n+1}$, which we imagine to be the origin in $\BR^{n+1}$ when we consider $\BH^{n+1}$ in the Poincar\'e model. In this model the sphere $\BS^n=\D_\infty\BH^{n+1}$ is, when endowed with the distance $d_o$, nothing but the round unit sphere $\BS^n$. The volume is thus, up to normalisation, just the spherical volume. To relax notation we will, unless it could cause confusion, drop the subscript $o$ when working with the metric associated to the base point, that is
\[
d(\theta,\eta)=d_o(\theta,\eta),\ \diam(X)=\diam_o(X)\text{, and }\vol(X)=\vol_o(X).
\]
We refer to metric balls $B=\{\eta \in \BS^n \vert d(\eta,\theta)<R\}\subset\BS^n$ with respect to $d=d_o$ as {\em round balls}. Given such a ball $B$ and a positive parameter $\lambda$, let $\lambda B=\{\eta \in \BS^n \vert d(\eta,\theta)<\lambda R\}$ be the round ball with same center and with radius scaled by $\lambda$. Incidentally, round balls are also metric balls, with different center and radius, with respect to all the other metrics $d_x$, but we stress that we reserve the adjective {\em round} to refer to the standard metric $d=d_o$.

Under a {\em round annulus} we understand an annulus $A\subset\BS^n$ of the form
\[
A=\{\eta\in\BS^n\vert r<d(\eta,\theta)<R\}\subset\BS^{n+1}
\]
for some $\theta\in\BS^{n+1}$ and $0<r<R<\pi$. Under the {\em dome of a round ball $B$} we understand nothing other than the convex hull of the said ball:
\[
\Dome(B)=\text{convex hull in }\BH^{n+1}\text{ of }B.
\]
The dome $\Dome(A) \subset \BH^{n+1}$ of a round annulus $A \subset \BS^n$ is defined similarly. 

Now, if $B\subset\BS^n$ is a round ball centered at $\theta$ of diameter $\diam B<\pi$ then we say that the {\em center of the boundary of the dome of $B$} is the point 
\[
\Center(\D\Dome(B))=\D\Dome(B)\cap[o,\theta],
\]
where the geodesic $[o,\theta]$ connecting $o$ and $\theta$ meets the boundary of the dome of $B$.

We will also encounter other annuli as well besides the round ones. In fact, if we are given a pair of distinct points $x,y\in\BH^n$ we denote by 
\begin{equation}\label{eq-conformal annulus}
A(x,y)=\{\theta\in\BS^n\vert\pi_{[x,y]}(\theta)\in(x,y)\}
\end{equation}
the set of points in the boundary at infinity which project to an interior point of the interval $[x,y]$. Note that $A(x,y)$ is a round annulus if and only if $o$ belongs to the geodesic containing the segment $[x,y]$.

\section{Harmonic extension}\label{sec-harmonic}

In Section \ref{sec-BCG} we recall the construction of Besson, Courtois, and Gallot's {\em barycentric extension} 
$$F_f:\BH^{n+1}\to\BH^{n+1}$$
of a non-constant quasiregular map $f:\BS^n\to\BS^n$. The map $F_f$ is smooth and extends $f$ continuously in the sense that for all $\theta\in\BS^n$ and all sequences $(x_m)$ in $\BH^{n+1}$ converging to $\theta$ we have
$$f(\theta)=\lim_{m\to\infty}F_f(x_m).$$
In Section \ref{sec-meat} we prove the following 3 facts about such barycentric extensions.

\begin{prop}\label{prop-lipschitz}
Let $F_f:\BH^{n+1}\to\BH^{n+1}$ be the barycentric extension of a non-constant $K$-quasiregular map $f:\BS^{n}\to\BS^{n}$ of degree $d\ge 1$. Then there is $L=L(K,d,n)>0$ satisfying 
$$\norm{DF_f(x)}, \norm{D^2F_f(x)}\le L$$
for all $x\in\BH^{n+1}$. Here $\Vert\cdot\Vert$ stands for the operator norm induced by the hyperbolic metric.
\end{prop}

\begin{prop}\label{prop-volume}
Let $F_f:\BH^{n+1}\to\BH^{n+1}$ be the barycentric extension of a non-constant $K$-quasiregular map $f:\BS^n\to\BS^n$ of degree $d\ge 1$. For every $\eta>0$ there are $\delta=\delta(\eta,K,d,n)>0$ and $R_0=R_0(\eta,K,d,n)>0$ having the following property: for every $x\in\BH^{n+1}$ and every round ball $B\subset\BS^n$ with diameter $\diam_{F_f(x)}(B)\le\delta$ we have
$$\vol_x\left(\left\{v\in T^1_x\BH^{n+1}\middle\vert F_f(\exp_x(Rv))\in\Dome(B)\right\}\right)<\eta$$
for all $R\ge R_0$. 
\end{prop}

\begin{prop}\label{prop-radialQI}
Let $F_f:\BH^{n+1}\to\BH^{n+1}$ be the barycentric extension of a non-constant $K$-quasiregular map $f:\BS^n\to\BS^n$ of degree $d\ge 1$. There is $c_0=c_0(n,K,d)>0$ such that for each $\varepsilon>0$ there is $R_0=R_0(n,K,d,\varepsilon)$ such that the set
$$Q_x = \{ v\in T^1_x\BH^n \vert d_{\BH^n}(F_f(x), F_f(\exp_x(Rv)) \ge c_0 R - c_0\ \text{for\ all}\ R\ge R_0\}$$
has measure $\vol_x(Q_x)\ge 1-\varepsilon$ for each $x\in \BH^n$.
\end{prop}

For the convenience of the reader we summarise briefly the content of these three propositions. Fix $K,d$ and $n$.
\begin{itemize}
\item Proposition \ref{prop-lipschitz} asserts that the maps $F_f$ are uniformly Lipschitz and have uniformly bounded second derivative.
\item Proposition \ref{prop-volume} states that the maps $F_f$ do not map sets of large visual volume to sets of small visual diameter. 
\item Finally, Proposition \ref{prop-radialQI} asserts that the maps $F_f$ behave like quasi-isometries in most directions.
\end{itemize}

Assuming these three proposition for the time being, we prove Theorem \ref{harmonic}. We rely heavily on the work of Benoist and Hulin \cite{Benoist-Hulin} and we highly recommend to the reader to first read that beautifully written paper.

\begin{proof}[Proof of Theorem \ref{harmonic}]
We start by fixing
\[
c=\max\left\{L,\frac 1{c_0}\right\},
\]
where $L$ comes from Proposition \ref{prop-lipschitz} and $c_0$ from Proposition \ref{prop-radialQI}. To avoid surcharging notation, let 
$$F=F_f:\BH^{n+1}\to\BH^{n+1}$$
be the barycentric extension of $f$. Also, we continue denoting by $o$ the base point of $\BH^{n+1}$. For $x\in\BH^{n+1}$ and $r>0$, let  
\[
\BB(x,r)=\{z\in\BH^{n+1}\vert d_{\BH^{n+1}}(z,x)\le r\}
\]
be the closed hyperbolic ball with center $x$ and radius $r$ and let $S(x,r)=\D\BB(x,r)$ the sphere with the same center and radius.

Consider for $R>0$ the restriction of $F$ to the sphere $S(o,R)$ 
and let 
$$h_R:\BB(o,R)\to\BH^{n+1}$$
be the solution on the ball $\BB(o,R)$ of the Dirichlet problem $h_R=F$ on $S(o,R)$. That is, $h_R$ is continuous on the closed ball, harmonic in the interior, and agrees with $F$ on the boundary. We show that the harmonic map $h_R$ is at uniformly bounded distance from the original map $F\vert_{\BB(o,R)}$:

\begin{claim}
There exists $\rho>0$ for which  
$$\rho_R=\max_{x\in\BB(o,R)}d_{\BH^{n+1}}(F(x),h_R(x))<\rho$$ 
for all $R>0$.
\end{claim}

This claim is the analogue of Proposition 3.6 in \cite{Benoist-Hulin}. In fact, if we assume the validity of the claim, then we may follow word-by-word the argument in \cite[Proof of Theorem 1.1]{Benoist-Hulin} to prove that, up to passing to a subsequence, the maps $h_R$ converge uniformly on compacta to a harmonic map which stays at bounded distance of $F$, and which hence extends continuously to our given map $f:\BS^n\to\BS^n$.

In other words, all we have to do is to prove the claim. To do so, we follow the strategy of \cite{Benoist-Hulin} and argue by contradiction. Suppose thus that there are values for $R$ for which $\rho_R$ is arbitrarily large. Using the same notation as in \cite{Benoist-Hulin}, let $x_R\in\BB(o,R)$ be a point satisfying $d_{\BH^{n+1}}(F(x_R),h_R(x_R))=\rho_R$, and set 
$$r_R=\sqrt[3]{\rho_R}.$$ 
By the choice of $c$, the first and second derivatives of $F$ are bounded by $c$. Now, the boundary estimate given by Proposition 3.8 in \cite{Benoist-Hulin} implies that
\[
\BB(x_R,r_R)\subset\BB(0,R-1)
\]
as long as
\[
r_R \le \frac{\rho_R}{16c^2(n+1)}.
\]
Since $r_R = \sqrt[3]{\rho_R}$, we this bound holds if 
\begin{equation}\label{eq-rho-bound1}
\rho_R>(16(n+1))^{3/2}c^3.
\end{equation}
From now on we assume \eqref{eq-rho-bound1} and work only on this smaller ball $\BB(x_R,r_R)$.

As in \cite[Section 4.1]{Benoist-Hulin} we write the (restriction) to $\BB(x_R,r_R)$  of the maps $F$ and $h=h_R$ in polar coordinates centered at the point $y_R=F(x_R)$, that is 
\[
F(z)=\exp_{y_R}(\rho_f(z)\cdot v_f(z))\text{ where }\rho_f(z)\in\BR_+\text{ and }v_f(z)\in T^1_{y_R}\BH^{n+1},
\]
\[
h(z)=\exp_{y_R}(\rho_h(z)\cdot v_h(z))\text{ where }\rho_h(z)\in\BR_+\text{ and }v_h(z)\in T^1_{y_R}\BH^{n+1},
\]
and consider the sets
$$U_R=\left\{v\in T^1_{x_R}\BH^{n+1}\middle\vert \rho_h(\exp_{x_R}(r_R\cdot v))\ge\rho_R-\frac 1{2c}r_R\right\}\text{ and}$$
$$V_R=\left\{v\in T^1_{x_R}\BH^{n+1}\middle\vert \rho_h(\exp_{x_R}(t\cdot v)\ge\frac 12\rho_R\text{ for all }t\in[0,r_R]\right\}.$$

The main ingredient in computations in \cite[Section 4.2]{Benoist-Hulin} -- besides the generic facts like the subharmonicity of the function $\rho_R$ or Green's theorem -- is the upper coarse Lipschitz estimate for the quasi-isometry. Thus the same computations with the estimate in Proposition \ref{prop-lipschitz} yield, for our sets $U_R$ and $V_R$, the volume estimates
\begin{equation}\label{eq-BHvolumeestimate}
\vol_{x_R}(U_R) \ge\frac 1{3c^2}\quad\text{and}\quad
\vol_{x_R}(V_R) \ge 1-2^{12}(n+1)c\cdot\frac{r_R^2}{\rho_R}
\end{equation}
for radii $R$ satisfying \eqref{eq-rho-bound1}. 

Also, the arguments in \cite[Section 4.4]{Benoist-Hulin} go through without a problem, showing that the angle between $v_h(x_R)$ and $v_h(\exp_{x_R}(r_R\cdot v))$ is very small for $v$ in $V_R$. More precisely, we have
\begin{equation}\label{eq-angle-bound1}
\angle(v_h(x_R),v_h(\exp_{x_R}(r_R\cdot v)))\le \frac {8\rho_R^2}{\sinh(\rho_R/4)}\text{ for all }v\in V_R.
\end{equation}
Again, \eqref{eq-angle-bound1} is valid for all $R$ satisfying \eqref{eq-rho-bound1}.

As the alert reader might have noticed, we went from Section 4.2 in \cite{Benoist-Hulin} to Section 4.4, skipping thus Section 4.3. It is indeed in Section 4.3 of \cite{Benoist-Hulin} where the authors make use of the fact that, up to a bounded additive error, quasi-isometries expand distances linearly. In our setting, the map $F$ fails to have this property, and this is what we have to repair. 

What saves us is that, by Proposition \ref{prop-radialQI}, $F$ behaves like a radial isometry in most directions. More concretely, setting $\varepsilon=\frac 1{6c^2}$, let $c_0$, $R_0$ and the set 
$$Q_{x_R} =\{ v\in T^1_{x_R}\BH^n \vert d_{\BH^n}(F(x_R), F(\exp_{x_R}(rv)) \ge c_0\cdot r - c_0\ \text{for\ all}\ r\ge R_0\}$$
be as provided by the said proposition. Also, we are interested in the case that
\begin{equation}\label{eq-rho-bound2}
r_R=\rho_R^{\frac 13}>R_0.
\end{equation}
Now, for $R$ satisfying \eqref{eq-rho-bound1} and \eqref{eq-rho-bound2}, set also
$$U_R'=U_R\cap Q_{x_R},$$
and note that \eqref{eq-BHvolumeestimate} and the choice of $\epsilon$ imply that
\begin{equation}\label{Iamsickofthis}
\vol_x(U_R')\ge \frac 1{6c^2}.
\end{equation}
As in Lemma 4.5 in \cite{Benoist-Hulin}, consider for $v\in U_R'$ the triangle with vertices 
$$p_1=y_R,\ p_2=F(\exp_{x_R}(r_R\cdot v)),\text{ and }p_3=h(\exp_{x_R}(r_R\cdot v)),$$ 
and note that we have the following bounds for the lengths of its sides:
\begin{equation*}
\begin{split}
d(p_2,p_3)&\le \rho_R\text{ the by definition of }\rho_R,\\
 d(p_3,p_1)&\ge\rho_R-\frac 1{2c}r_R\text{ because }v\in U_R'\subset U_R,\text{ and}\\
 d(p_1,p_2)&\ge c_0\cdot r_R - c_0\text{ because }v\in U_R'\subset Q_{x_R}.
 \end{split}
 \end{equation*}
Hyperbolic trigonometry (see also \cite[Lemma 2.1]{Benoist-Hulin}) implies that the angle 
$$\angle(v_F(\exp_{x_R}(r_R\cdot v)),v_h(\exp_{x_R}(r_R\cdot v)))=\angle([p_1,p_2],[p_1,p_3])$$
is bounded by
$$\angle([p_1,p_2],[p_1,p_3])\le 4e^{-\frac{1}4(-d(p_2,p_3)+d(p_3,p_1)+d(p_1,p_2))}\le 4e^{-\frac 14
(-\frac 1{2c}r_R+c_0 r_R - c_0)}.$$
Taking into account that $c>\frac 1{c_0}$, we get
\begin{equation}\label{eq-angle-bound2}
\angle(v_F(\exp_{x_R}(r_R\cdot v)),v_h(\exp_{x_R}(r_R\cdot v)))\le  4e^{-\frac 14(\frac 1{2c}r_R - c_0)}
\end{equation}
for every $R$ satisfying \eqref{eq-rho-bound1} and \eqref{eq-rho-bound2}, and for every $v\in U_R'$.

Combining \eqref{eq-angle-bound1} and \eqref{eq-angle-bound2} we get 
\begin{equation}\label{eq-angle-bound3}
\angle(v_h(x_R),v_F(\exp_{x_R}(r_R\cdot v)))\le \frac {8\rho_R^2}{\sinh(\rho_R/4)}+4e^{-\frac 14(\frac 1{2c}r_R - c_2)}
\end{equation}
for all $R$ satisfying \eqref{eq-rho-bound1} and \eqref{eq-rho-bound2} and for all $v\in V_R\cap U_R'$. Since $r_R=\rho_R^{1/3}$, we have that the quantity on the right hand side of \eqref{eq-angle-bound3} tends to $0$ as $\rho_R \to \infty$. On the other hand, we have that 
\begin{equation}\label{eqlargevolume}
\vol_x(V_R\cap U_R')\ge\frac 1{6c^2}-2^{12}(n+1)c\cdot\frac{r_R^2}{\rho_R}
\end{equation}
by \eqref{eq-BHvolumeestimate} and \eqref{Iamsickofthis}.

Taken altogether this means that, if $\rho_R$ is very large, the set $V_R\cap U_R'$ of uniformly positive measure \eqref{eqlargevolume} is mapped by $v\mapsto F(\exp_{x_R}(R\cdot v))$ into a set with visually very small diameter \eqref{eq-angle-bound3}. Proposition \ref{prop-volume} asserts that this is not possible, proving that that $\rho_R$ is uniformly bounded from above. We have proved, with invaluable help of \cite{Benoist-Hulin}, the claim and thus the theorem.
\end{proof}

\section{Quasiregular maps}\label{sec-quasiregular}

In this section we recall some basic facts about quasiregular maps, about the relation between the conformal geometry of $\BS^n$ and the metric geometry of $\BH^{n+1}$, and prove a compactness result for quasiregular maps which might find use in other settings. The latter is the only new result in this section. 

\subsection{Basic}
We recall basic facts and definitions about quasiregular maps and refer to \cite{Rickman-book}, specially Chapter I therein, for a detailed discussion. See also V\"ais\"al\"a's beautiful ICM survey \cite{VaisalaICM} for a pleasant exposition of classical results on quasiregular maps.

Recall from the introduction that a continuous mapping $f\colon \BS^n \to \BS^n$ is \emph{quasiregular} if $f$ is in the Sobolev space $W^{1,n}(\BS^n, \BS^n)$ and there exists a constant $K\ge 1$ for which $f$ satisfies the quasiconformality condition
\begin{equation}
\label{eq:qr}
\norm{Df(x)}^n \le K \det Df(x)
\end{equation}
for almost every $x\in \BS^n$. In particular, quasiregular homeomorphisms are quasiconformal maps. 

\begin{bem}
A comment on terminology. The minimal constant $K$ in \eqref{eq:qr} is called the \emph{outer distortion of $f$} and denoted $K_O(f)$. The minimal constant $K\ge 1$ satisfying
\[
\det Df(x) \le K_I \min_{|v|=1} |Df(x) v|^n 
\]
for almost every $x\in \BS^n$ is called the \emph{inner distortion $K_I(f)$ of $f$}. The constants $K_O(f)$ and $K_I(f)$ are related by elementary estimates $K_O(f) \le K_I(f)^{n-1}$ and $K_I(f)\le K_O(f)^{n-1}$ and the constant
\[
K(f) = \max\{ K_O(f), K_I(f)\}
\]
is called the \emph{distortion of $f$}. Although sometimes in the literature quasiregular mappings satisfying \eqref{eq:qr} are called \emph{$K$-quasiregular mappings}, we reserve this term to a quasiregular mappings $f\colon \BS^n\to \BS^n$ satisfying $K(f)\le K$ as is done in \cite{Rickman-book}. 
\end{bem}

Quasiregular mappings need not be locally injective. We have, however, by theorems of Reshetnyak, that non-constant quasiregular mappings are (generalized) branched covers, in the sense that they are discrete and open. This means that pre-images of points are discrete sets, and that open sets are mapped to open sets.

\begin{fact}
Non-constant quasiregular maps are discrete and open.
\end{fact}

Moreover, if $f\colon \BS^n \to \BS^n$ is a non-constant quasiregular map, then by the Chernavskii--V\"ais\"al\"a theorem, its {\em branch set $B_f$}, that is the set of points where $f$ fails to be a local homeomorphism, has topological dimension at most $n-2$. In fact, the sets $f(B_f)$ and $f^{-1}(f(B_f))$ also have topological dimension at most $n-2$. It follows that none of these closed sets separate, and in fact the restriction 
$$f|_{\BS^n\setminus f^{-1}(f(B_f))} \colon \BS^n \setminus f^{-1}(f(B_f)) \to \BS^n \setminus f(B_f)$$ 
is a covering map. More importantly for us, the branch set $B_f$ has vanishing Lebesgue measure.

\begin{fact}
The branch set $B_f$ of a non-constant quasiregular map $f:\BS^n\to\BS^n$ is a closed set of topological dimension at most $n-2$ satisfying $\vol(B_f)=0$.
\end{fact}

The sets $f(B_f)$ and $f^{-1}(f(B_f))$ also have vanishing measure. This follows from the fact we just stated together with the facts that the push-forward of Lebesgue measure under a non-constant quasiregular map is absolutely continuous with respect to Lebesgue measure.

\begin{fact}
The push-forward measure $f_*\vol$ of the Lebesgue $n$-measure $\vol$ on $\BS^n$ under a non-constant quasiregular map is absolutely continuous with respect to $\vol$, meaning that $\vol(f(E))=0$ if and only if $\vol(E)=0$ for any measurable subset $E\subset\BS^n$.
\end{fact}

Returning to the topological properties of non-constant quasiregular maps, we record that, whenever $f:\BS^n\to\BS^n$ is such a map  then each point $x\in\BS^n$ has a {\em normal neighborhood} $U \subset \BS^n$ with respect to $f$, by which we mean that $f(\partial U)=\partial(f(U))$ and $f^{-1}(f(x)) \cap \overline{U} = \{x\}$, where $\overline U$ is the closure of $U$.

\begin{fact}
If $f\colon \BS^n \to \BS^n$ is a non-constant quasiregular map then for every $y\in \BS^n$ there is $r_f(y)>0$ with the property that, for $r\in (0,r_f(y)]$, the components $U_1,\ldots, U_k$ of $f^{-1}(B(y,r))$ are normal neighborhoods. Moreover, for each domain $U_i$ the local degree function $B(y,r)\to \BZ$, $y'\mapsto \deg(y',f,U_i)$, is constant.
\end{fact}

Since, by another theorem of Reshetnyak (see \cite[Theorem I.4.5]{Rickman-book}), a non-constant quasiregular mapping is sense-preserving, we also have that if $\deg(y, f, U_i) =1$ then $f|_{U_i}$ is a local homeomorphism. Recall that a quasiregular homeomorphism $f:U\to V$ between open sets of the sphere is also known as a {\em quasiconformal} map.

\subsection{Compactness}

Besides the generalities we just mentioned, we use in the sequel a compactness result for quasiregular maps of fixed degree mentioned in the introduction. This is our next goal.

\begin{named}{Proposition \ref{prop-compactness}}
Let $n\ge 2$, $K\ge 1$, and $d\in \BN$. Then, given a sequence $(f_m)$ of $K$-quasiregular mappings $\BS^n \to \BS^n$ of degree at most $d$, there exists a set $P\subset \BS^n$ of cardinality at most $d$ and a subsequence $(f_{m_k})$ converging locally uniformly on $\BS^n\setminus P$ to a $K$-quasiregular mapping $f\colon \BS^n \to \BS^n$ of degree at most $d$. 
\end{named}

Although we do not need this terminology in what follows, this compactness property states that this family of maps is in fact \emph{quasinormal}; see e.g.\;Pang, Nevo, and Zalcman \cite{Pang-Nevo-Zalcman} for terminology.

\begin{proof}[Proof of Proposition \ref{prop-compactness}]
We show first that, given such a sequence $(f_m)$, there exists a finite set $P$ of uniformly bounded cardinality satisfying the wanted property. After this we show that the cardinality of $P$ is actually bounded by $d$. Note that we may assume that $d\ge 1$, since quasiregular maps of degree zero are constant maps.

For each $\ell\in \BN$ let $\CB_\ell$ be a finite covering of $\BS^n$ by round balls of radius $2^{-\ell}$ and satisfying that $\frac 12B\cap\frac 12B'=\emptyset$ for all distinct $B,B'\in\CB_\ell$. Note that since the balls $\frac 12B$ with $B\in \CB_{\ell}$ are pair-wise disjoint, a point in $\BS^n$ belongs to at most $4^n$ balls in $\CB_{\ell}$.

Let $\delta>0$ be a parameter which we will let to tend to zero at the end of the proof. Now, for $m\in \BN$, let $\CB_{\ell,m}(\delta)\subset \CB_\ell$ be the subcollection consisting of balls $B\in \CB_\ell$ whose image $f_m(B)$ contains the complement of a ball of radius $\delta$ in $\BS^n$. We are going to bound the cardinality $|\CB_{\ell,m}(\delta)|$ of this collection of balls 
independently of $\ell$ and $m$. To do so note that 
\begin{equation}\label{borel-cantelli}
\sum_{B\in \CB_{\ell,m}(\delta)} \vol(f_m(B)) \ge (1-c\cdot\delta^n)\cdot \vert\CB_{\ell,m}(\delta)\vert
\end{equation}
for some $c$ depending only on the dimension. Recall that $\vol$ is the spherical volume of $\BS^n$, normalized in such a way that $\vol(\BS^n)=1$.

It follows from \eqref{borel-cantelli} that there is a point $p\in\BS^n$ which belongs to at least $(1-c\cdot \delta^n)\cdot\vert\CB_{\ell,m}(\delta)\vert$ images of balls $B\in\CB_{\ell,m}(\delta)$. Taking into account that each point belongs to at most $4^n$ balls in $\CB_{\ell,m}(\delta)\subset\CB_{\ell}$, we obtain that $p$ has at least $4^{-n}\cdot(1-c\cdot \delta^n)\vert\CB_{\ell,m}(\delta)\vert$ distinct preimages. Since $f_m$ is quasiregular, we conclude that $\vert f_m^{-1}(p)\vert \le \deg f_m \le d$,
which means that $\frac {1-c\cdot\delta^n}{4^n}\vert\CB_{\ell,m}(\delta)\vert\le d$ and thus that
\begin{equation}\label{eq-cardinality bound}
\vert\CB_{\ell,m}(\delta)\vert\le \frac{4^n\cdot d}{1-c\cdot\delta^n}.
\end{equation}
We have established the desired bound.

At this point we might assume that, after passing to a subsequence, the sets $\CB_{\ell,m}(\delta)$ are independent of $\delta$, meaning that there exists $\CB_\ell(\delta)\subset\CB_\ell$ for which
$$\CB_{\ell,m}(\delta)=\CB_\ell(\delta)$$
for all $m$. The uniform cardinality bound \eqref{eq-cardinality bound} implies also that, up to passing to a further subsequence, we can assume that, when $\ell\to\infty$, the sets $\bigcup \CB_{\ell}(\delta)$ converge in the Hausdorff topology to a finite set $P \subset \BS^n$. 

Let now $U\subset \BS^n$ be an open set whose closure $\overline{U}$ does not meet the set $P$. For each $\ell\in \BN$, let $\CB_\ell(U)$ be the collection of balls in $\CB_\ell$ which meet $U$. Since $P$ and $U$ have positive distance we get that there exists an index $\ell_0\in \BN$ with the property that the union of the balls in $\CB_\ell(U)$ is disjoint of the balls in $\CB_\ell(\delta)$ for all $\ell\ge\ell_0$. Now, for each $B\in \CB_{\ell_0}(U)$ we have that, with only finitely many $m$'s as exceptions, the set $f_m(B)$ is contained in the complement of a ball of radius $\delta$ in $\BS^n$. Thus the sequence $f_m|_B$ is equicontinuous (\cite[Corollary III.2.7]{Rickman-book}) and has a subsequence tending uniformly to a $K$-quasiregular map $f_B \colon B \to \BS^n$ (\cite[Theorem VI.8.6]{Rickman-book}). A standard diagonal argument now yields a subsequence $(f_{m_j})$ tending locally uniformly on $\BS^n\setminus P$ to a $K$-quasiregular map $f \colon \BS^n\setminus P \to \BS^n$. 

Before bounding the cardinality of $P$ by $d$, we show that $f$ is at most $d$-to-1 and thus extends to a $K$-quasiregular map $\BS^n\to \BS^n$ \cite[Corollary III.2.11]{Rickman-book}. To see that this is the case, let $y\in \BS^n$ and let $x_1,\ldots, x_k$ be distinct points in $f^{-1}(y)$. Since $f$ is discrete and open, there exists (pair-wise disjoint) normal neighborhoods $U_1,\ldots, U_k$ of points $x_1,\ldots, x_k$, that is, domains compactly contained in $\BS^n\setminus P$ for which $f(\partial U_i)= \partial(f(U_i))$ and $f^{-1}(y) \cap \overline{U_i} = \{x_i\}$ for each $i=1,\ldots, k$. Since the maps $(f_{m_j})$ converge locally uniformly to $f$, we have that $f_{m_j}^{-1}(y) \cap U_i \ne \emptyset$ for each $i=1,\ldots, k$ by the upper semi-continuity of the local index; see e.g.\;\cite[Lemma VI.8.13]{Rickman-book}. This means that $f_{m_j}^{-1}(y)$ has at least $k$ distinct preimages. This implies that $k\le d$ because $f_m$ is a quasiregular map of degree $d$. We have proved that $f$ is at most $d$-to-one.
 
It remains to show that $\# P \le d$. Given $\ell_1\ge 0$ and $m_1\ge 0$, there exists $\ell\ge \ell_1$, $m\ge m_1$, and a collection $\{B_x\}_{x\in P} \subset \CB_{\ell,m}(\delta)$ a collection of pair-wise disjoint balls satisfying $x\in B_x$ for each $x\in P$. Now, the same argument as above shows that there exists a point $y\in \BS^n$ which meets at least $(1-c_n\delta^n)|\CB_{\ell,m}(\delta)|$ of the images $f_m(B_x)$. Since this time the balls $B_x$ are pair-wise disjoint, we have that 
\[
|P| = |\CB_{\ell,m}(\delta)| \le \frac{1}{1-c_n \delta^n} |f_m^{-1}(y)| \le \frac{d}{1-c_n \delta^n} \to d
\]
as $\delta \to 0$. The proof is complete.
\end{proof} 

In this note we will use Proposition \ref{prop-compactness} over and over again. Something else that we will use also all the time is the relation between conformal geometry on $\BS^n$ and the metric geometry on $\BH^{n+1}$, or more precisely between moduli of annuli in $\BS^n$ and distances of points in $\BH^{n+1}$. We recall this relation next, proving only facts which are well-known but (perhaps) hard to find explicitly in the literature.

\subsection{Modulus and hyperbolic distances}
Recall that the \emph{conformal modulus of a path family} of a family of paths $\Gamma$ in a $n$-dimensional Riemannian manifold $M$ is
\[
\Mod(\Gamma) = \inf_{\rho} \int_M \rho^n \vol_M,
\]
where the infimum is taken over all non-negative Borel functions $\rho$ on $M$ satisfying
\[
\int_\gamma \rho\ \text{ds} \ge 1
\]
for all locally rectifiable paths $\gamma\in \Gamma$; see \cite[Chapter 1]{Vaisala-book} or \cite{Heinonen-book}. We are only interested in the classical cases that $M$ is the sphere $\BS^n$, euclidean space $\BR^n$, or open sets thereof. 

The transformation formula implies that a $K$-quasiconformal homeomorphism $f\colon M \to N$ satisfies for every path family $\Gamma$ in $M$ the inequality
\begin{equation}
\label{eq:QC_Gamma}
\frac{1}{K} \Mod(\Gamma) \le \Mod(f(\Gamma)) \le K \Mod(\Gamma),
\end{equation}
where $f(\Gamma) = \{ f\circ \gamma \colon \gamma \in \Gamma\}$. In fact, a homeomorphism satisfying \eqref{eq:QC_Gamma} is actually quasiconformal, see \cite[Chapter 2]{Vaisala-book}.

We are mostly interested in the modulus of certain curve families associated to (topological) annuli $A\subset\BR^n$. If we denote by $\D_1A$ and $\D_2A$ the two boundary components of $A$, let $\Gamma_A$ be the set of all paths contained in $A$ and connecting $\D_1A$ and $\D_2A$. We then set
$$\Mod(A)=\Mod(\Gamma_A)$$
and refer to this quantity as the {\em modulus of $A$}. For reasons that will be clear shortly, round annuli 
$$A_{r,R}=\{x\in\BR^n\vert r\le\vert x\vert\le R\}$$
will be specially relevant. In this case the modulus is given by 
\begin{equation}\label{eq-modulus-round}
\Mod(A_{r,R})= \omega_{n-1}\left( \log\frac{R}{r}\right)^{1-n},
\end{equation}
for some $\omega_{n-1}$ depending only on the dimension. Note in particular that the modulus is small if the ratio $R/r$ is large, and large if the said ratio is close to $1$.

\begin{bem}
In fact, as the reader probably suspected, $\omega_{n-1}$ is the volume of the unit sphere of dimension $n-1$ with respect to the measure induced form the Riemannian metric. This measure is a multiple of $\vol(\cdot)$, namely $\omega_{n-1}\vol(\cdot)$.
\end{bem}

As we mentioned earlier, the modulus of round annuli in $\BS^n$ is closely related to distances in $\BH^{n+1}$:

\begin{lem}\label{lem:distance-modulus}
Given a round open annulus $A\subset\BS^n=\D\BH^{n+1}$ let $U$ and $V$ be the two connected components of $\BH^{n+1}\setminus\Dome(A)$. Then we have
$$\Mod(A)\ge \omega_{n-1}\cdot d_{\BH^{n+1}}(x,y)^{1-n}$$
for all $x\in U$ and $y\in V$. Moreover, equality holds for if and only if $x$, $y$ realise the distance between $\D U $ and $\D V$.
\end{lem}

This lemma is well-known and we just prove it here for convenience of the reader.

\begin{proof}
We work in the upper half space model $\BH^{n+1}\subset\BR^{n+1}$. We can thus move our round annulus around using M\"obius 
transformations so that $A=A_{r,1}$ for a suitable choice of $r<1$. The boundaries of the dome $\Dome(A)$ are then the euclidean hemispheres $\BS^n\cap\BH^{n+1}$ and $r\BS^n\cap\BH^{n+1}$. This two hemispheres are at distance $-\log(r)$ with respect to the hyperbolic distance. The claim follows now from the formula for the modulus given above.
\end{proof}

Lemma \ref{lem:distance-modulus} asserts that there is a quantitative relation between the conformal geometry of $\BS^n$ and metric geometry of $\BH^{n+1}$. However, for this lemma to be really useful we need to work with round annuli instead of general topological annuli. The following lemma asserts that, as long as we are sufficiently far away from the branch set, the image of a round annulus of small modulus under a quasiregular map contains a round annulus of comparable modulus.

\begin{lem}\label{modulus-nicerlemma}
Let $f\colon \BS^n \to \BS^n$ be a non-constant $K$-quasiregular map and $B=B(x,R)\subset \BS^n$ a ball of radius $R<1/2$ for which the restriction $f|_B \colon B \to f(B)$ is quasiconformal. Then there exists $\lambda_0=\lambda_0(n,K)$ and $C_1=C_1(n,K)\ge 1$ with the property that if $A\subset \frac 14 B$ is a round annulus with $\Mod(A)\le\lambda_0$ then $f(A)$ contains a round annulus $A'$ satisfying
$$\frac{1}{C_1} \Mod(A)\le \Mod(A') \le C_1 \Mod(A).$$
\end{lem}
\begin{proof}
Note that up to post- and pre-composition with stereographic projections we can replace $\BS^n$ by $\BR^n$ and suppose that $f(0)=0$. In fact, the number $4$ in $\frac 14B$ is chosen in such a way that this conformal transformation can be chosen so that the following two conditions are satisfied at the same time:
\begin{itemize}
\item $A=\{x\in\BR^n\vert s\le\vert x\vert\le t\}$ for suitable choices of $0<s<t<\frac 12$, and
\item $B$ contains the unit ball $B(0,1)=\{x\in\BR^n,\ \vert x\vert<1\}$.
\end{itemize}
Being quasiconformal on the ball $B(0,1)$, the restriction of $f$ to $B(0,\frac 12)$ has the quasisymmetry property (see \cite[Chapter 11]{Heinonen-book}, specially Theorems 11.3 and 11.14). This means more concretely that there are $C>0$ and $\alpha\ge 1$ depending only on the quasiconformally constant $K$ and on the dimension $n$ with 
$$\frac{|f(y)|}{|f(z)|} \le C\cdot\max\left\{\left( \frac{|y|}{|z|}\right)^\alpha,\left( \frac{|y|}{|z|}\right)^{1/\alpha}\right\}$$
for all $y,z\in B(0,\frac 12)\setminus\{0\}$. At this point we are ready to determine $\lambda_0$. Indeed, suppose that the sphere $S(0,r)=\{x\in\BR^n \,\vert\, \vert x\vert=r\}$ meets both $f(S(0,s))$ and $f(S(0,t))$ at points $y\in S(0,r) \cap f(S(0,s))$ and $z\in S(0,r) \cap f(S(0,t))$. Then, by quasisymmetry,
$$1 = \frac{|f(y)|}{|f(z)|} \le  C \left(\frac{s}{t}\right)^{1/\alpha}$$
and thus $t \le C^{\alpha} s$. This means that $\Mod(A)\ge\omega_{n-1}\left(\log C^{\alpha}\right)^{1-n}$. Thus, as long as 
$$\Mod(A)<\lambda_0\stackrel{\text{def}}=\min\left\{1,\omega_{n-1}\left(\log C^{\alpha}\right)^{1-n}\right\}$$ 
there exists a round annulus contained in the topological annulus $f(A)$. 

Supposing now that $\Mod(A)<\lambda_0$, let $A'=A_{r,R}$ be the maximal round annulus in $f(A)$. Then there exist $y\in S(0,r) \cap f(S(0,s))$ and $z\in S(0,R) \cap f(S(0,t))$. The same computation as above yields
$$\frac{R}{r} = \frac{|f(z)|}{|f(y)|} \le C\left( \frac{t}{s} \right)^{\alpha}\quad\text{ and }\quad\frac{r}{R} = \frac{|f(y)|}{|f(z)|} \le C\left( \frac{s}{t} \right)^{1/\alpha}.$$
Hence
$$\frac 1C\left(\frac ts\right)^{1/\alpha}\le \frac Rr\le C\left(\frac ts\right)^\alpha.$$
The claim follows now from \eqref{eq-modulus-round}. 
\end{proof}

\section{Barycentric extension}\label{sec-BCG}

In this section we recall the construction of the Besson--Courtois--Gallot barycentric extension. We refer to \cite{BCG} for details of the construction.

\subsection{Barycenter}
First recall the definition of the Busemann function
\[
B_o(\cdot,\theta):\BH^{n+1}\to\BR,\ \ B_o(y,\theta)=\lim_{p\to\theta}\left(d(y,p)-d(o,p)\right)
\]
centered at $\theta\in\D_\infty\BH^{n+1}=\BS^n$ and normalized at the base point $o\in\BH^{n+1}$. Given a probability measure $\mu$ on $\BS^n$, consider the function
\[
\CB_\mu:\BH^{n+1}\to\BR,\ \ y\mapsto \int_{\BS^n}B_o(y,\theta)d\mu(\theta).
\]
Busemann functions are convex and thus $\CB_\mu$ is convex as well. In fact, if for example $\mu$ is atomless, then $\CB_\mu$ is also strictly convex and proper. It follows that $\CB_\mu$ has a unique minimum, the {\em barycenter} $\BCG(\mu)$ of $\mu$. This construction is natural in the sense that 
$$\BCG(g_*\mu)=g(\BCG(\mu))$$
for any $g\in\Isom(\BH^{n+1})$. Note also that uniqueness of the barycenter implies that, as long as the measures remain atomless, the barycenter depends continuously on the measure $\mu$, where we have endowed the space of measures with the usual weak-$\ast$-topology.

Our next aim is to prove that the barycenter satisfies what we refer to as the {\em gravity principle}: \emph{if most of the measure is in northern Sweden, then the barycenter is in a Nordic country.} Before making this precise, recall that the Busemann function $B_o(\cdot,\theta)$ is smooth and that its gradient at $y$ is the vector
\[
\DD B_o(\cdot,\theta)\vert_y=-n_\theta(y),
\]
where $n_\theta(y)$ is as earlier the unit vector in $T_y\BH^{n+1}$ pointing to $\theta$, that is
$$\lim_{R\to\infty}\exp_y(R\cdot n_\theta(y))=\theta.$$
Smoothness of the Busemann functions and the dominated convergence theorem imply that $\CB_\mu(\cdot)$ is also smooth and that 
\begin{equation}\label{eq-bcg-implicit1}
y=\BCG(\mu)\text{ if and  only if }\DD\CB_\mu\vert_y=-\int_{\BS^n}n_\theta(y)d\mu(\theta)=0.
\end{equation}
We are ready to establish the gravity principle.

\begin{lem}\label{lem:china}
Let $\mu$ be a probability measure on $\BS^n$ and $B\subset\BS^n$ a round ball for which $\mu(B)>\frac 23$. Then $d_{\BH^{n+1}}(\BCG(\mu), \Dome(B)) < 1$.
\end{lem}
\begin{proof}
Let $x\in\BH^{n+1}$ be a point for which $d_{\BH^{n+1}}(x,\Dome(B))\ge 1$. We have to show that $x\neq\BCG(\mu)$. 

By elementary hyperbolic geometry, $d_{\BH^{n+1}}(x,\Dome(B))\ge 1$ implies that $\diam_x(B)<\frac\pi 2$. 
Let $\zeta$ be the center of the ball $B$ with respect to the visual metric $d_x$. Then $\langle n_\theta(x),n_\zeta(x)\rangle_x\ge\frac 12$ for all $\theta\in B$, where $\langle\cdot,\cdot\rangle_x$ is the (hyperbolic) scalar product on $T_x\BH^{n+1}$. We also have $\langle n_\theta(x),n_\zeta(x)\rangle_x\ge -1$ for a general $\theta\in\BS^n$. 
Thus
\[
\left\langle\int_{\BS^n}n_\theta(x)d\mu(\theta),n_\zeta(x)\right\rangle\ge\frac 12\mu(B)-\mu(\BS^n\setminus B)>\frac 12\cdot \frac 23-\frac 13=0.
\]
It follows from \eqref{eq-bcg-implicit1} that $x\neq\BCG(\mu)$. 
\end{proof}

\subsection{The Besson--Courtois--Gallot map}\label{subsec-bcg}
We are now ready to define the {\em barycentric extension}
$$F_f:\BH^{n+1}\to\BH^{n+1}$$
of a non-constant quasi-regular map $f:\BS^n\to \BS^n$. Given $x\in\BH^{n+1}$, we push $\vol_x$ forward using $f$ and define
$$F_f(x)=\BCG(f_*(\vol_x));$$
recall that $\vol_x$ is the visual probability measure associated to $x$.
This makes sense since $f_*\vol_x$ is absolutely continuous with respect to Lebesgue measure and thus atomless.

We recall some properties of the map $F_f$. First, the barycentric extension $F_f$ is defined in metric terms and thus behaves well with respect to isometries. More concretely, we have
$$F_{h\circ f\circ g}(x)=h(F_f(g(x)))$$
 for all $x\in\BH^{n+1}$ and all $g,h\in\Isom(\BH^{n+1}).$

Second, note that the measure $f_*(\vol_x)$ depends continuously on $x$. This implies that $F_f$ is continuous. Moreover, the gravity principle (Lemma \ref{lem:china}) implies that the map $F_f:\BH^{n+1}\to\BH^{n+1}$ extends continuously to the original map $f:\BS^n\to\BS^n$. 

The barycentric extension $F_f$ in fact smooth, independently of the regularity of $f$. To see that this is the case identify, for example using the disk model, $\BH^{n+1}$ with an open subset of $\BR^{n+1}$ and thus consider tangent vectors to $\BH^{n+1}$ as vectors in $\BR^{n+1}$. We may then consider the map
\[
G_f:\BH^{n+1}\times\BH^{n+1}\to\BR^{n+1},\ \ G_f(x,y)=-\DD \CB_{f_*(\vol_x)}\vert_y=\int_{\BS^n}n_\theta(y)df_*(\vol_x)(\theta).
\]
From \eqref{eq-bcg-implicit1} we get an implicit equation
\[
F_f(x)=y\text{ if and only if }G_f(x,y)=0.
\]
Given that visual measures $\vol_x$ are absolutely continuous to each other with Radon-Nikodym derivative
\[
\frac{d\vol_x}{d\vol_o}(\theta)=e^{-(n-1)B_o(x,\theta)}
\]
for $x\in \BH^{n+1}$ and $\theta\in \BS^n$, 
we get from the transformation formula the following alternative form 
\begin{equation}\label{eq-Gf}
G_f(x,y)=\int_{\BS^n}n_{f(\theta)}(y)e^{-(n-1)B_o(x,\theta)}d\vol_o(\theta).
\end{equation}
for $G_f$.
Smoothness of the individual functions $(x,y)\mapsto n_{f(\theta)}(y)e^{-(n-1)B_o(x,\theta)}$ implies, together with the dominated convergence theorem, that $G_f$ is smooth. Moreover, positivity of the Hessian of $y\mapsto\CB_{f_*(\vol_x)}(y)$ means that $\frac{D G_f}{d y}$ is non-singular. It follows thus from the implicit function theorem that $F_f$ is smooth, and that its derivatives are computed in terms of the derivatives of $G_f$. For instance, the first and second derivatives are given by
\begin{equation*}\label{eq-dev1}
\begin{split}
D F_f\vert_x(v)&=-\left(\frac{D G_f}{dy}\Big\vert_{F_f(x)}\right)^{-1}\left(\frac{\D G_f}{dx}\Big\vert_x(v)\right) \text{ and }\\
D^2 F_f\vert_x(v)&=-\left(\frac{D G_f}{dy}\Big\vert_{F_f(x)}\right)^{-1}\left(\frac{D^2G_f}{dx^2}\Big\vert_x(v)+\frac{D^2G_f}{dy^2}\Big\vert_{F_f(x)}(DF_f\vert_x(v))\right).
\end{split}
\end{equation*}

\subsection{Compactness}
We will prove Proposition \ref{prop-lipschitz}, Proposition \ref{prop-volume}, and Proposition \ref{prop-radialQI} in the next section. The proofs rely, in one way or another, on the following key observation. 

\begin{prop}\label{key-lemma}
Let $n\ge 2$, $K\ge 1$, and $d\ge 1$. Suppose $(f_m)$ is a sequence of $K$-quasiregular mappings $\BS^n \to \BS^n$ of degree at most $d$ and whose barycentric extensions satisfy $F_{f_m}(o)=o$. Then there exists a set $P\subset \BS^n$ of cardinality at most $d$ and a subsequence $(f_{m_k})$ converging locally uniformly on $\BS^n\setminus P$ to a non-constant $K$-quasiregular mapping $f\colon \BS^n \to \BS^n$ of degree at most $d$ which also satisfies $F_f(o)=o$.
\end{prop}
\begin{proof}
From Proposition \ref{prop-compactness} we get that there is a set $P\subset \BS^n$ of cardinality at most $d$ such that, up to passing to a subsequence, the maps $f_m$ converge locally uniformly on $\BS^n\setminus P$ to a $K$-quasiregular mapping $f\colon \BS^n \to \BS^n$ of degree at most $d$. It remains to show that $f$ is not constant and that $F_f(o)=o$, but this is easily done. Convergence of $f_m$ to $f$ on $\BS^n\setminus P$ implies that the push-forward measures converge, that is, 
\begin{equation}
\label{eq-measure-limit}
\lim_{m\to\infty}(f_m)_*\vol=f_*\vol.
\end{equation}

If the limit map $f$ is constant, the measures $(f_m)_*\vol$ converge to a Dirac measure centred at a point $\theta$. Then the gravity principle (Lemma \ref{lem:china}) implies that the points $F_{f_m}(o)=\BCG((f_m)_*\vol)$ converge to $\theta$, contradicting our assumption. Thus $f$ is non-constant.

By the convergence of measures $(f_m)_*\vol$, the sequence of barycenters $F_{f_m}(o)=\BCG((f_m)_*\vol)$ converges to the barycenter $F_f(o)=\BCG(f_*\vol)$. It follows that $F_f(o)=o$. 
\end{proof}

We may use Proposition \ref{key-lemma} to obtain information on quasiregular maps $\BS^n \to \BS^n$ themselves. A first observation along those lines, and one that we are actually going to use below, is that whenever $f$ is a non-constant quasiregular map whose barycentric extension fixes $o$, a large portion of the domain of the map can be covered by decent sized balls on which $f$ restricts to a quasiconformal map:

\begin{lem}
\label{lem:balls}
Let $\varepsilon>0$, $K\ge 1$ and $d\ge 1$. Then there exists $\delta>0$ such that whenever $f:\BS^n\to\BS^n$ is a non-constant $K$-quasiregular map of degree $\deg(f)\le d$ and satisfying $F_f(o)=o$, then there exists a collection $\CB= \{B_1,\ldots, B_r\}$ of round balls in $\BS^n$ with $\vol(\bigcup_i \frac{1}{100}B_i)\ge 1-\varepsilon$ and satisfying for each $i$ that $\diam(B_i)\ge \delta$ and that the restriction $f|_{B_i}$ is quasiconformal. 
\end{lem}

\begin{proof}
Seeking a contradiction suppose that for some $\varepsilon$, $K$ and $d$ there is a sequence $(f_m)$ of non-constant $K$-quasiregular maps of degree $d$ and satisfying $F_{f_m}(o)=o$, and for which there is no $\delta$ for which a desired covering exists. Passing to a subsequence we might assume by Proposition \ref{key-lemma} that the maps $f_m$ converge to a non-constant quasiregular map $f$. Now, the branch set $B_f$ is closed and has measure $0$, which means that there is a compact set $C\subset\BS^n\setminus B_f$ with $\vol(C)\ge 1-\epsilon$. Now, each point in $C$ is contained in some round ball $B'$ such that $f$ is quasiconformal on the ball $101B'$ of hundred and one times the radius. Compactness of $C$ implies that it can be covered by finitely many such balls $B_1,\dots,B_r$. Now, since $f\vert_{101B_i'}$ is quasiconformal, we deduce that $f_m\vert_{100B_i'}$ is quasiconformal for large $m$ and all $i$. It now suffices to set $B_i=100B_i'$ and $\delta=\min\{\diam(B_i)\,\vert\, i=1,\dots,r\}$ to obtain the desired contradiction.
\end{proof}

\section{Proofs of the propositions}\label{sec-meat}
As the title hopefully suggests, in this section we prove Proposition \ref{prop-lipschitz}, Proposition \ref{prop-volume}, and Proposition \ref{prop-radialQI}. Working like little ants, we deal with them one by one.

\subsection{Lipschitz bound}
We prove first that the operator norms of the first and second derivatives of the barycentric extension of a non-constant $K$-quasiregular map of degree $d$ are bounded by quantities that depend only on $K$, $d$ and the dimension. More precisely we prove Proposition \ref{prop-lipschitz}:

\begin{named}{Proposition \ref{prop-lipschitz}}
Let $F_f:\BH^{n+1}\to\BH^{n+1}$ be the barycentric extension of a non-constant $K$-quasi-regular map $f:\BS^{n}\to\BS^{n}$ of degree $d\ge 1$. Then there is $L=L(K,d,n)>0$ satisfying 
$$\norm{DF_f(x)}, \norm{D^2F_f(x)}\le L$$
for all $x\in\BH^{n+1}$. Here $\Vert\cdot\Vert$ stands for the operator norm induced by the hyperbolic metric.
\end{named}
\begin{proof}
We prove the claim for $\norm{DF_f(x)}$ -- the proof for the norm of the second derivative is analogous. Arguing by contradiction suppose that, for some $K$ and $d$, there is a sequence $(f_m)$ of $K$-quasiregular maps of degree $d$ and a sequence of points $(x_m)$ in $\BH^{n+1}$ such that $\norm{DF_{f_m}(x_m)}\to\infty$. Consider then sequences $(g_m)$ and $(h_m)$ of isometries of $\BH^{n+1}$ satisfying
$$g_m(o)=x_m\text{ and }h_m(F_{f_m}(x_m))=o$$
for each $m\in \BN$. As we mentioned above, we have  
$$F_{h_m\circ f_m\circ g_m}=h_m\circ F_{f_m}\circ g_m.$$
Since $g_m$ and $h_m$ are isometries, we have that
$$\norm{DF_{h_m\circ f_m\circ g_m}(o)}=\norm{DF_{f_m}(g_m(o))}=\norm{DF_{f_m}(x_m)}\to\infty$$
as $m\to \infty$. This means that, up to replacing $f_m$ by $h_m\circ f_m\circ g_m$ we can assume that $x_m=o$ and that $F_{f_m}(o)=o$ for all $m$. We do this from now on.

Proposition \ref{key-lemma} asserts now that there is a set $P$ containing at most $d$ points and such that, up to passing to a subsequence, the restrictions of the maps $f_m$ to $\BS^n\setminus P$ converge locally uniformly to a non-constant quasiregular map $f:\BS^n\to\BS^n$ with $F_f(o)=o$. 

The dominated convergence theorem implies now that the maps $G_{f_i}:\BH^{n+1}\times\BH^{n+1}\to\BR^{n+1}$, given in \eqref{eq-Gf}, and their derivatives, converge then to the map $G_f$ and its derivative. It follows thus from the formula (given at the end of Section \ref{subsec-bcg}) for the differential of the barycentric extensions that $D F_{f_i}\vert_o$ converges to $D F_f\vert_o$. This contradicts that assumption we made at the very beginning of the proof and thus concludes the proof of Proposition \ref{prop-lipschitz}.
\end{proof}

\subsection{Bound on volume contraction}
Using a similar argument we prove also that the barycentric extension of a $K$-quasiregular map of degree $d$ maps sets of large visual volume to sets of large visual diameter. This is the content of Proposition \ref{prop-volume}:

\begin{named}{Proposition \ref{prop-volume}}
Let $F_f:\BH^{n+1}\to\BH^{n+1}$ be the barycentric extension of a non-constant $K$-quasi-regular map $f:\BS^n\to\BS^n$ of degree $d\ge 1$. For every $\eta>0$ there are $\delta=\delta(\eta,K,d,n)>0$ and $R_0=R_0(\eta,K,d,n)>0$ such that for every $x\in\BH^{n+1}$ and every round ball $B\subset\BS^n$ with diameter $\diam_{F_f(x)}(B)\le\delta$ we have
$$\vol_x\left(\left\{v\in T^1_x\BH^{n+1}\middle\vert F_f(\exp_x(Rv))\in\Dome(B)\right\}\right)<\eta$$
for all $R\ge R_0$. 
\end{named}
\begin{proof}
Arguing again by contradiction suppose that there are $\eta>0$, a sequence $(f_m)$ of $K$-quasiregular maps of degree $d$, a sequence of points $(x_m)$ in $\BH^{n+1}$, and a sequence of round balls $B_m\subset\BS^n$ satisfying 
$$\diam_{F_{f_m}(x_m)}(B_m)\to 0$$
while
$$\vol_{x_m}\left(\left\{v\in T^1_{x_m}\BH^{n+1}\middle\vert F_f(\exp_{x_m}(mv))\in\Dome(B_i)\right\}\right)\ge \eta.$$
Choosing as above sequences $(g_m)$ and $(h_m)$ of isometries of $\BH^{n+1}$ satisfying
$$g_m(o)=x_m\text{ and }h_m(F_{f_m}(x_m))=o,$$
we may replace the maps $f_m$ by $h_m\circ f_m\circ g_m$ and suppose 
$$x_m=o\text{ and }F_{f_m}(o)=o$$
from the very beginning. Now, as in the proof of Proposition \ref{prop-lipschitz}, we may apply Proposition \ref{key-lemma} and pass to a subsequence to assume that the maps $f_m$ converge, off a set $P$ having at most $d$ points, to a non-constant quasi-regular map $f$. This implies that whenever 
$$U\subset\bar\BH^{n+1}=\BH^{n+1}\cup\BS^n$$
is an open neighborhood of the set $P$ then the maps $F_{f_i}$ converge to $F_f$ uniformly on $\bar\BH^{n+1}\setminus U$. We may choose $U$ in such a way that the following two properties are satisfied:
\begin{itemize}
\item $\vol\left(\left\{v\in T^1_o\BH^{n+1}\middle\vert \exp_o(v)\in U\right\}\right)<\frac 12\eta$, and
\item for $R\ge 1$ and $v\in T^1_o\BH^{n+1}$, we have $\exp_o(Rv)\in U$ if and only if $\exp_o(v)\in U$.
\end{itemize}
Note that, by the second property, we may identity $\BS^n\cap U$ with 
$$\left\{v\in T^1_o\BH^{n+1}\middle\vert \exp_o(v)\in U\right\}$$ 
and we do so from now on.

Now, by the uniform convergence of the restrictions of $F_{f_m}$ to the restriction of $F_f$ on $\BH^{n+1}\setminus U$ implies that the maps 
\[
\phi_m:\BS^n\setminus(\BS^n\cap U)\to\bar\BH^{n+1},\quad v\mapsto F_{f_m}(\exp_o(m\cdot v)),
\]
converge uniformly to the restriction $f \colon \BS^n \setminus (\BS^n\cap U) \to \BS^n$ of $f$ as $m\to \infty$.

Since the preimages of points under the map $f$ have at most $d$ preimages, it follows that preimages under $\phi_m$ of sets of small visual diameter are contained in the union of at most $d$ sets of small visual diameter, and thus have small volume. It follows that there exists $\delta>0$ for which, for all balls $B$ of diameter $\delta$, we have
$$\vol(\{v\in T_o^1\BH^{n+1}\vert\exp_o(v)\in U\text{ and } \phi_m(v)\in\Dome(B)\})<\frac 12\eta.$$
Since also the set $\{v\in T^1_o\BH^{n+1}\vert \exp_o(v)\in U\}$ has volume less than $\frac 12\eta$, we obtain that
$$\vol(\{v\in T_o^1\BH^{n+1}\vert\phi_m(v)\in\Dome(B)\})<\eta.$$
This contradicts our assumption and thus proves Proposition \ref{prop-volume}.
\end{proof}

\subsection{Radial quasi-isometry property}

And now the grand finale, Proposition \ref{prop-radialQI}.

\begin{named}{Proposition \ref{prop-radialQI}}
Let $F_f:\BH^{n+1}\to\BH^{n+1}$ be the barycentric extension of a non-constant $K$-quasiregular map $f:\BS^n\to\BS^n$ of degree $d\ge 1$. There is $c_0=c_0(n,K,d)>0$ such that, for each $\varepsilon>0$, there is $R_0=R_0(n,K,d,\varepsilon)$ such that the set
$$Q_x = \{ v\in T^1_x\BH^n \vert d_{\BH^n}(F_f(x), F_f(\exp_x(Rv)) \ge c_0 R - c_0\ \text{for\ all}\ R\ge R_0\}$$
has measure $\vol_x(Q_x)\ge 1-\varepsilon$ for each $x\in \BH^n$.
\end{named}

Before going any further note that it is easy to see that the constant $R_0$ has to actually depend on $\epsilon$. In the course of the proof of Proposition \ref{prop-radialQI} we make use of the following fact:

\begin{lem}\label{distance-modulus}
Let $B\subset\BS^n$ be a round ball of at most radius $\frac 12$ and let $x_1=\Center(\D\Dome(B))$ be the center of the boundary of the dome of $B$. For any $x_2\in\Dome(\frac 1{25}B)$ there is a round annulus $A\subset B$ satisfying
$$\Mod(A)= \omega_{n-1}(d_{\BH^{n+1}}(x_1,x_2)-2)^{1-n}$$
and whose dome $\Dome(A)$ separates $x_1$ from $x_2$. Moreover, if $B_1$ and $B_2$ are the two connected components of $\BS^n\setminus A$, then $\vol_{x_i}(B_i)\ge\frac 23$ for $i=1,2$.
\end{lem}

Recall that the {\em center of the boundary of the dome of $B$} is the point $\Center(\D\Dome(B))\in\D\Dome(B)$ through which runs the geodesic ray from the base point $o$ to the spherical center of $B$. Note also that the number $25$ has been chosen in the statement to guarantee that the distance between the two involved points $x_1$ and $x_2$ is relatively large, namely $d_{\BH^{n+1}}(x_1,x_2)\ge\log(25)\sim 3.218875\ldots>2$. After these comments we can start the proof.

\begin{proof}[Proof of Lemma \ref{distance-modulus}]
Consider the arc $[o,x_2]$ oriented as departing $o$ and arriving to $x_2$. Let $x_1'\in[o,x_2]$ be minimal such that the annulus $A(x_1',x_2)$ as in \eqref{eq-conformal annulus} is contained in $B$. Elementary hyperbolic geometry shows that 
\[
d_{\BH^{n+1}}(x_1,x_1')<\frac 1{10}.
\]
Now choose $[\bar x_1,\bar x_2]\subset[x_1',x_2]$ for which
\[
d_{\BH^{n+1}}(\bar x_1,x_1)=d_{\BH^{n+1}}(\bar x_2,x_2)=1.
\]
Let $A=A(\bar x_1,\bar x_2)$; this is a round annulus because $\bar x_1,\bar x_2$ and $o$ are colinear. Elementary hyperbolic geometry implies then again that the $\vol_{x_i}(B_i)\ge \frac 23$, where $B_1$ and $B_2$ are as in the statement. Since the round annulus $A$ has modulus 
$$\Mod(A)=\omega_{n-1}d_{\BH^{n+1}}(\bar x_1,\bar x_2)^{1-n},$$ 
the claim follows.
\end{proof}

We are now ready to obtain the main estimate needed in the proof of Proposition \ref{prop-radialQI}.

\begin{lem}\label{blahblahblah}
For $K\ge 1$ and $n\ge 2$, there are $C>0$ and $c>0$ for which the following holds: Let $f:\BS^n\to\BS^n$ be a $K$-quasiregular map, $F_f:\BH^{n+1}\to\BH^{n+1}$ the barycentric extension of $f$, and $B\subset\BS^n$ a round ball of at most radius $\frac 12$ having the property that $f\vert_B:B\to f(B)$ is quasiconformal. Let also $x_1=\Center(\D\Dome(B))$ be the center of the dome of $B$, and $x_2\in\Dome(\frac 1{100}B)$ another point. Then
\[
d_{\BH^{n+1}}(F_f(x_1),F_f(x_2))\ge C\cdot d_{\BH^{n+1}}(x_1,x_2)-c.
\]
\end{lem}

The constant $100$ is chosen to be the product of $25$ from Lemma \ref{distance-modulus} and $4$ from Lemma \ref{modulus-nicerlemma}. 

\begin{proof}
Let $B'=\frac 14B$ be the ball of radius one fourth of the radius of $B$ and let $x_1'=\Center(\D\Dome(B'))$ be the center of the boundary of the dome of $B'$. To unify notation we set $x_2'=x_2$ and note that $d_{\BH^{n+1}}(x_1,x_1')\le\log(4)<2$. Since $\frac 1{100}B=\frac 1{25}B'$, we get from Lemma \ref{distance-modulus} a conformally round annulus $A\subset B'$ which 
satisfies 
$$\omega_{n_1}d_{\BH^{n+1}}(x_1',x_2')^{1-n}\le \Mod(A)\le \omega_{n_1}(d_{\BH^{n+1}}(x_1',x_2')-2)^{1-n}$$
and such that $\Dome(A)$ separates $x_1'$ from $x_2'$. Moreover, if 
$B_1$ and $B_2$ are the connected components of $\BS^n\setminus A$ labelled such that $B_2\subset B'\subset B$, then we also have $\vol_{x_i'}(B_i)\ge\frac 23$ for $i=1,2$.

Now, by the choice of $B'=\frac 14B$, we may apply Lemma \ref{modulus-nicerlemma}. Thus, with the constants $\lambda_0$ and $C_1$ as in the said lemma, we obtain that the topological annulus $f(A)$ contains a round annulus $\hat A$ satisfying 
$$\frac{1}{C_1} \Mod(A)\le \Mod(\hat A) \le C_1 \Mod(A),$$
as long as $\Mod(A)\le \lambda_0$. 

Let $\hat B_1$ and $\hat B_2$ be the two connected components of $\BS^n\setminus\hat A$, label-ed in such a way that $f(B_i)\subset\hat B_i$ and note that we have
$$(f_*\vol_{x_i'})(\hat B_i)\ge \vol_{x_i'}(B_i)\ge\frac 23$$
for $i=1,2$. The gravity principle (Lemma \ref{lem:china}) implies that  
$$d_{\BH^{n+1}}(F_f(x_i'),\Dome(\hat B_i))\le 1,$$
which implies, via Lemma \ref{distance-modulus}, that
$$d_{\BH^{n+1}}(F_f(x_1'),F_f(x_2'))\ge \sqrt[n-1]{\omega_{n-1}\cdot\Mod(\hat A)^{-1}}-2.$$
Taking into account that 
$$\Mod(\hat A)\le C_1\cdot \Mod(A)\le C_1\cdot \omega_{n-1}(d_{\BH^{n+1}}(x_1,x_2)-2)^{1-n}$$
we thus get that
$$d_{\BH^{n+1}}(F_f(x_1'),F_f(x_2'))\ge \frac 1{C_1^{\frac 1{n-1}}}d_{\BH^{n+1}}(x_1',x_2')-4.$$
Taking now into account that $F_f$ is $L$-Lipschitz by Proposition \ref{prop-lipschitz}, that $x_2=x_2'$ and that $d_{\BH^{n+1}}(x_1,x_1')<2$ we get that
$$d_{\BH^{n+1}}(F_f(x_1),F_f(x_2))\ge C\cdot d_{\BH^{n+1}}(x_1,x_2)-c$$
for $C=C_1^{\frac{-1}{n-1}}$ and $c=4-2L-2C_1^{\frac{-1}{n-1}}$. We have proved the lemma.
\end{proof}

We are now ready to prove Proposition \ref{prop-radialQI}.

\begin{proof}[Proof of Proposition \ref{prop-radialQI}]

First thing to do is to choose isometries $g$ and $h$ of $\BH^{n+1}$ with $g(o)=x$ and $h(F_f(x))=o$ and replace the $K$-quasiregular map $f$ by the $K$-quasiregular map $h\circ f\circ g$ of the same degree $d=\deg(f)$. In this way we might assume, as we do from now on, that $F_f(o)=o$. 

By Lemma \ref{lem:balls}, there exist $\delta=\delta(K,d,n,\epsilon)>0$ and a collection $\CB= \{B_1,\ldots, B_r\}$ of balls in $\BS^n$ for which $\vol(\bigcup_i \frac{1}{100}B_i)\ge 1-\frac 12\varepsilon$ and which satisfy, for each $i$, that $\frac12\ge \diam(B_i)\ge \delta$ and that the restriction $f|_{B_i}$ is quasiconformal. The volume bound implies that there exists $R_0=R_0(\delta)$ satisfying 
$$\vol\left(\left\{ v\in T^1_x\BH^n \middle\vert \exp_x(Rv)\in\bigcup_i\Dome\left(\frac{1}{100}B_i\right)\right\}\right)\ge 1-\varepsilon$$
for all $R\ge R_0$. 

Let now $x_i=\Center(\D\Dome(B_i))$ be the center of the boundary of the dome of $B_i$ for $i=1,\dots,r$ and note that the distance from these points to the base point $o$ is bounded in terms of a function of $\delta$, that is, 
\[
d_{\BH^{n+1}}(o,x_i)\le\alpha(\delta)\sim -\log(\delta).
\]
Now, Lemma \ref{blahblahblah} and the Lipschitz property of $F_f$ (Proposition \ref{prop-lipschitz}) shows that 
\[
d_{\BH^{n+1}}(F_f(o),F_f(x))\ge C\cdot d_{\BH^{n+1}}(o,x)-(c+(C+1)\cdot\alpha(\delta))
\]
for all $x\in\bigcup_i \Dome(\frac{1}{100}B_i)$. This proves the claim.
\end{proof}

\section{Dimension $n=1$}

The exclusion of dimension $n=1$ in Theorem \ref{harmonic} stems from the ill-suitedness of the definition of quasiregular mappings to that case. Indeed, for $n=1$, the condition 
\[
\norm{Df} \le K \det Df
\]
reduces to the condition $|du| \le K du$, where $du$ is the derivative of a function $u$ satisfying $f(\theta) = e^{iu(\theta)}$. Thus we only obtain that $du>0$, which implies that $f$ is an orientation preserving covering map $\BS^1 \to \BS^1$. 

To obtain a class of mappings with better compactness properties, we pass from a quasiconformality condition to a quasisymmetry condition. Recall that a homeomorphism $h \colon \BS^1\to \BS^1$ is \emph{$\eta$-quasisymmetric}, where $\eta \colon [0,\infty) \to [0,\infty)$ is a homeomorphism, if 
\begin{equation}
\label{qs}
\frac{d(f(x),f(y))}{d(f(x),f(z))} \le \eta\left( \frac{d(x,y)}{d(x,z)}\right)
\end{equation}
for all triples $x,y,z$ of distinct points in $\BS^1$.

In this spirit, we say that $f\colon \BS^1\to \BS^1$ is an \emph{$\eta$-quasisymmetric covering} if $f$ is a composition of an $\eta$-quasisymmetry $\BS^1\to \BS^1$ post-composed by the covering map $z\mapsto z^d$. Although, more intrinsic definition definitely has its virtues, note that, for example, the natural localization of the quasisymmetry condition \eqref{qs} leads to this definition. In addition, we immediately benefit from the good compactness properties of quasisymmetric homeomorphisms and obtain a counterpart of Lemma \ref{prop-compactness}.

\begin{lem}
\label{n=1:prop-compactness}
Let $\eta\colon [0,\infty) \to [0,\infty)$ be a homeomorphism and $d\in \BN$. Then, given a sequence $(f_m)$ of $\eta$-quasisymmetric covering mappings $\BS^1\to \BS^1$ of degree at most $d$, there exists a subsequence $(f_{m_k})$ converging uniformly on $\BS^1\setminus P$ to an $\eta$-quasisymmetric covering $f\colon \BS^1 \to \BS^1$ of degree at most $d$. 
\end{lem}

\begin{proof}
Factor $f_m = (z\mapsto z^d) \circ h_m$ for each $m\in \BN$. Since the family $\{h_m\}$ is $\eta$-quasisymmetric homeomorphisms is normal, we conclude that there exists a subsequence $(h_{m_k})$ converging uniformly to an $\eta$-quasisymmetric homeomorphism $h \colon \BS^1\to \BS^1$. Now we may take $f = (z\mapsto z^d) \circ h$.
\end{proof}

Regarding other results in Section \ref{sec-quasiregular}, we notice that assumption $n=1$ renders them obsolete. Lemma \ref{lem:distance-modulus} becomes the remark that for $x_1, x_2\in \BH^2$ on the opposite sides of the dome $\Dome(A(r,R))$ we have $d_{\BH^2}(x_1,x_2) \ge \log(R/r)$. The modulus estimates in Lemma \ref{modulus-nicerlemma} on the other hand are replaced by distance estimates given by function $\eta$. We leave formulation of this version to the interested reader. 

Having these facts -- and the fact that $f_*\vol$ is atomless -- at our disposal, it is easy to observe that the rest of the arguments of in the proof of Theorem \ref{harmonic} hold also for $n=1$. Theorem \ref{harmonic} thus holds for $n=1$ in the following form.

\begin{sat}
Let $f\colon \BS^1 \to \BS^1$ be an $\eta$-quasisymmetric covering. Then there exists harmonic map $H_f \colon \BH^2 \to \BH^2$ extending $f$.
\end{sat}


\end{document}